\documentclass[10pt,a4paper]{article}
\usepackage{verbatim}
\usepackage{amsmath}
\usepackage{amssymb}
\usepackage{amsthm}
\usepackage{amscd}
\usepackage{indentfirst}
\usepackage[top=25.4mm, bottom=25.4mm, left=31.7mm, right=32.2mm]{geometry}
\theoremstyle{definition}
\newtheorem{thm}{Theorem}
\newtheorem{lem}{Lemma}

\newtheorem{prop}{Proposition}
\newtheorem{cor}{Corollary}
\newtheorem{de}{Definition}
\newtheorem{rem}{Remark}

\numberwithin{equation}{section}

\begin{document}
\setlength{\abovedisplayskip}{1ex} 
\setlength{\belowdisplayskip}{1ex}

\title{The equivalence relationship between Li-Yorke $\delta$-chaos and
distributional $\delta$-chaos in a sequence
\footnote{Received date: 2009-03-09. Supported by NNSFC(10771079) and Guangzhou Education Bureau (08C016).
\newline \hbox{}\quad\ \ 
This is the English version of a published paper,
cite should be as  ``Jian Li and Feng Tan, \textit{The equivalence relationship between Li-Yorke $\delta$-chaos and
distributional $\delta$-chaos in a sequence}, Journal of South China Normal University (Natural Science Edition), 2010, No.3: 34--38 (In Chinese)''.}}
\author{Jian  Li,\quad Feng Tan \\ E-mail: lijian09@mail.ustc.edu.cn,\quad tanfeng@scnu.edu.cn\\
(School of Mathematics, South China Normal University, \\Guangzhou, China 510631)}
\date{}
\maketitle

\begin{center}
\begin{minipage}{13cm}
\footnotesize{ \textbf{Abstract: } In this paper, we discuss the
relationship between Li-Yorke chaos and distributional chaos in a
sequence. We point out the set of all distributional $\delta$-scramble
pairs in the sequence $Q$ is a $G_\delta$ set, and prove that Li-Yorke
$\delta$-chaos is equivalent to distributional $\delta$-chaos in a
sequence, a uniformly chaotic set is a distributional scramble set in some
sequence and
a class of transitive system implies distributional chaos in a sequence. \\
\textbf{Keywords: } Li-Yorke chaos, distributional chaos in a sequence, transitive system\\
\textbf{MSC(2000): } 54D20, 37B05}
\end{minipage}
\end{center}

\section{Introduction}
Throughout this paper a {\em topological dynamical system} (TDS for short) is a pair $(X,f)$,
where $X$ is a non-vacuous compact metric space with metric $d$ and $f$ is a continuous map for $X$ to itself.
If $Y\subset X$ is a closed invariant set (i.e.\@ $f(Y)\subset Y$),
then we call $(Y,f|_Y)$ is a {\em subsystem} of $(X,f)$.

Let $A$ be a subset of $X$, denote the {\em closure} of $A$ by $\overline{A}$.
For a given positive number $\delta$,
put $[A]_\delta=\{x\in X\mid \inf_{y\in A}d(x,y)<\delta\}$.
Denote the {\em diagonal} of the product space $X\times X$
by $\Delta=\{(x,x)\in X\times X\mid x\in X\}$.

In \cite{LiYorke1975}, Li and Yorke first introduced the word ``chaos" to
describe the complexity of the orbits.
\begin{de}
Let $(X,f)$ be a TDS. A pair $(x,y)\in X\times X$ is called a {\em Li-Yorke scrambled pair}, if
$$\liminf_{n\rightarrow\infty}d(f^n(x),f^n(y))=0, \quad\limsup_{n\rightarrow\infty}d(f^n(x),f^n(y))>0.$$
A subset $C$ of $X$ is called a {\em Li-Yorke chaotic set}, if every $(x,y)\in
C\times C\setminus\Delta$ is a Li-Yorke scrambled pair.
The system $(X,f)$ is called {\em Li-Yorke chaotic}, if there exists some uncountable Li-Yorke chaotic set.

For a give positive number $\delta$,
a pair $(x,y)\in X\times X$ is called a {\em Li-Yorke $\delta$-scrambled pair}, if
$$\liminf_{n\rightarrow\infty}d(f^n(x),f^n(y))=0,\quad\limsup_{n\rightarrow\infty}d(f^n(x),f^n(y))>\delta.$$
A subset $C$ of $X$ is called a {\em Li-Yorke $\delta$-chaotic set}, if every $(x,y)\in
C\times C\setminus\Delta$ is a Li-Yorke $\delta$-scrambled pair.
The system $(X,f)$ is called {\em Li-Yorke $\delta$-chaotic},
if there exists some uncountable Li-Yorke $\delta$-chaotic set.
\end{de}

In \cite{Simtal1994}, Schweizer and Smital introduced a new kind of chaos,
which is usually called distribution chaos.
Later, the authors in \cite{Du} introduced the conception of distribution chaos in a sequence.
See \cite{Yang} and \cite{Gu} for recent results.

Let $(X,f)$ be a TDS and $Q=\{m_i\}_{i=1}^{\infty}$ be a strictly increasing sequence of positive integers.
For $x,y\in X$, $t>0$ and $n\geq 1$, put
$$\Phi^n_{(xy,Q)}(t)=\frac{1}{n}\#\{1\le i\le n\mid d(f^{m_i}(x),f^{m_i}(y))\le t\}, $$
where $\#\{\cdot\}$ denote the cardinal number of a set. Let
$$\Phi_{(xy,Q)}(t)=\liminf_{n\to\infty}\Phi^n_{(xy,Q)}(t),\quad
\Phi^\star_{(xy,Q)}(t)=\limsup_{n\to\infty}\Phi^n_{(xy,Q)}(t).$$
then $\Phi_{(xy,Q)}$ and $\Phi^\star_{(xy,Q)}$ are called the {\em lower and
upper distribution function of $(x,y)$ with respect to the sequence $Q$}.
Clearly, for every $t>0$, $\Phi_{(xy,Q)}(t)\le\Phi^\star_{(xy,Q)}(t)$.

\begin{de}\label{1}
Let $(X,f)$ be a TDS and $Q$ be a strictly increasing sequence of positive integers.
A pair $(x,y)\in X\times X$ is called a
{\em distributional scrambled pair in the sequence $Q$}, if \\
\indent(1) for every $t>0$, $\Phi^{\star}_{(xy,Q)}(t)=1$; \\
\indent(2) there exists some $s>0$ such that $\Phi_{(xy,Q)}(s)=0$. \\
\noindent A subset $D$ of $X$ is called a
{\em distributional chaotic set in the sequence $Q$},
if every $(x,y)\in D\times D\setminus\Delta$ is a distributional scrambled pair
in the sequence $Q$.
The system $(X,f)$ is called {\em distributional chaotic in a sequence},
if there exists some strictly increasing sequence of positive integers $P$
such that there is an uncountable distributional chaotic set in the sequence $P$.

For a give positive number $\delta$, a pair $(x,y)\in X\times X$ is called a {\em distributional
$\delta$-scrambled pair in the sequence $Q$}, if \\
\indent(1) for every $t>0$, $\Phi^{\star}_{(xy,Q)}(t)=1$; \\
\indent(2) $\Phi_{(xy,Q)}(\delta)=0$. \\
\noindent A subset $D$ of $X$ is called a
{\em distributional $\delta$-chaotic set in the sequence $Q$},
if every $(x,y)\in D\times
D\setminus\Delta$ is a distributional $\delta$-scrambled pair in the sequence $Q$.
The system $(X,f)$ is called {\em distributional $\delta$-chaotic in a sequence},
if there exists some strictly increasing sequence of positive integers $P$
such that there is an uncountable distributional $\delta$-chaotic set in the sequence $P$.
\end{de}

Recently, the Li-Yorke and distributional chaos have aroused great interest.
For a continuous map $f$ from the unit closed interval $[0,1]$ to itself,
if $f$ has a period point with periodic $3$,
then the system $([0,1],f)$ is Li-Yorke chaotic \cite{LiYorke1975}.
The system $([0,1],f)$ is Li-Yorke chaotic
if and only if it is distributional chaotic in a sequence \cite{Du}.
If a system $(X,f)$ has positive entropy then it is Li-Yorke chaotic \cite{Blanchard}.
A weakly mixing system is distributional chaotic in a sequence \cite{Yang}.

In this paper, we discuss the relationship between Li-Yorke chaos and distributional chaos in a sequence and
chaos properties in transitive systems. The organization of the paper is as follows.
In section 2, we give a equivalent definition of distributional scrambled pair and
point out the set of all distributional $\delta$-scramble pairs in the sequence $Q$ is a $G_\delta$ set.
In section 3, we first show that a countable Li-Yorke chaotic set (resp. Li-Yorke $\delta$-chaotic set)
is also a  distributional chaotic set in some sequence
(resp. distributional $\delta$-chaotic set in some sequence).
Then using the well-known Mycielski theorem we show that Li-Yorke $\delta$-chaos
is equivalent to distributional $\delta$-chaos in a sequence.
In section 4, we focus on the transitive systems.
We prove that a uniformly chaotic set is a distributional scramble set in some
sequence and
a class of transitive system implies distributional chaos in a sequence.

\section{The structure of the set of all distributional $\delta$-scrambled pairs in the sequence $Q$}

Let $Q=\{m_i\}_{i=1}^\infty$ be a strictly increasing sequence of positive integers and
$P$ be a subsequence of $Q$. The upper limit
$$\limsup_{k\rightarrow\infty}\frac{\#\{P\cap\{m_1,\cdots,m_k\}\}}{k}$$
is called the {\em upper density of $P$ with respect to $Q$}, denoted by $\overline{d}(P\mid Q)$

For every $a\in [0,1]$, put $\overline{\mathcal M}_Q(a)=\{P\subset Q\mid
P \text{ is infinite and }\overline{d}(P\mid Q)\geq
a\}$

Let $(X,f)$ be a TDS and $U\subset X$ be a nonempty open set.
For every $a\in [0,1]$, define
$$\mathcal{F}(U,Q,\overline{\mathcal M}_Q(a))=\{x\in X\mid N(x,U,Q)\in\overline{\mathcal M}_Q(a)\},$$
where $N(x,U,Q)=\{m\in Q\mid f^m(x)\in U\}$.

Following the proof of Theorem 3.2 in \cite{Xiong2007}, we have
\begin{prop}\label{Gset}
Let $(X,f)$ be a TDS and $Q$ be a strictly increasing sequence of positive integers.
Then for every $a\in[0,1]$ and nonempty open set $U\subset X$,
$\mathcal{F}(U,Q,\overline{\mathcal M}_Q(a))$ is a $G_\delta$ set.
\end{prop}

Note that the metric in the product space $X\times X$ is denoted by $d^2$, i.e.\@ for every
$(x_1,x_2),(y_1,\allowbreak y_2)\in X\times X$, $d^2((x_1,x_2),(y_1,y_2))=\max\{d(x_1,y_1),d(x_2,y_2)\}$.
Now we can easily get a equivalent definition of distributional scrambled pair in the sequence $Q$.

\begin{prop}\label{2}
Let $(X,f)$ be a TDS and $Q$ be a strictly increasing sequence of positive integers.
Then $(x,y)\in X\times X$ is a distributional scrambled pair in the sequence $Q$ if and only if \\
\indent (1) for every $\epsilon>0$, $(x,y)\in
\mathcal{F}([\Delta]_\epsilon,Q,\overline{\mathcal M}_Q(1));$\\
\indent (2) there exists some $t>0$, such that $(x,y)\in \mathcal{F}(X\times
X\setminus\overline{[\Delta]_t},Q,\overline{\mathcal M}_Q(1)).$

\noindent For a given $\delta>0$, $(x,y)\in X\times X$
is a distributional $\delta$-scrambled pair in the sequence $Q$  if and only if \\
\indent (1) for every $\epsilon>0$, $(x,y)\in
\mathcal{F}([\Delta]_\epsilon,Q,\overline{\mathcal M}_Q(1));$\\
\indent (2) $(x,y)\in \mathcal{F}(X\times
X\setminus\overline{[\Delta]_\delta},Q,\overline{\mathcal M}_Q(1)).$
\end{prop}

Combining Proposition \ref{Gset} and \ref{2}, we have

\begin{prop}\label{Gdeltaset}
Let $(X,f)$ be a TDS, $Q$ be a strictly increasing sequence of positive integers and $\delta>0$.
Then the set of all distributional $\delta$-scrambled pairs in the sequence $Q$
is a $G_\delta$ subset of $X\times X$.
\end{prop}

%%We omit the proof here, if you has any question, please email to the author.

\section{Li-Yorke~$\delta$-chaos is equivalent to distributional $\delta$-chaos in a sequence}

Let $(X,f)$ be a TDS. A subset $C$ of $X$ is called a {\em Cantor set},
if it is homeomorphic to the standard Cantor ternary set.
A subset $A$ of $X$ is called a {\em Mycielski set}, if it is a union of countable Cantor sets.

\begin{thm}[Mycielski\cite{Mycielski}]\label{thm:Mycielski-cor}
Let a $X$ be a complete second countable metric space without isolated points.
If $R$ is a dense $G_\delta$ subset of $X\times X$, then there exists some
dense Mycielski subset $K\subset X$ such that $K\times K\setminus\Delta\subset R$.
\end{thm}

\begin{lem}\cite{Gu}\label{SEQ}
Let $\{S_i\}_{i=1}^{\infty}$ be a sequence of strictly increasing sequences of positive integers.
Then there exists a strictly increasing sequence $Q$ of positive integers
such that $\overline{d}(S_i\cap Q\mid Q)=1$ for every $i\geq 1$.
\end{lem}

\begin{thm}\label{countable}
Let $(X,f)$ be a TDS and $\delta>0$. If $C\subset X$ is a countable Li-Yorke chaotic set
(resp.\@ Li-Yorke $\delta$-chaotic set),
then there exists a strictly increasing sequence $Q$ of positive integers
such that $C$ is a distributional chaotic set in the sequence $Q$
(resp.\@ distributional $\delta$-chaotic set in the sequence $Q$).
\end{thm}
\begin{proof}
Let $C=\{x_i\in X:\, i=1,2,\ldots\}$. By the definition of Li-Yorke scrambled pair,
for every $i\neq j$, there exists $\delta_{ij}>0$ such that
\[\liminf_{n\to\infty}d(f^n(x_i), f^n(x_j))=0,\quad
\limsup_{n\to\infty}d(f^n(x_i), f^n(x_j))> \delta_{ij}.\]
i.e., there are two strictly increasing sequences
$P_{i,j}=\{n_k^{i,j}\}_{k=1}^\infty$ and $S_{i,j}=\{m_k^{i,j}\}_{k=1}^\infty$ such that
\[\lim_{k\to\infty}d(f^{n_k^{i,j}}(x_i), f^{n_k^{i,j}}(x_j))=0,\quad
\lim_{k\to\infty}d(f^{m_k^{i,j}}(x_i), f^{m_k^{i,j}}(x_j))>\delta_{ij}.\]
By Lemma \ref{SEQ}, there exists a strictly increasing sequence $Q$ such that
\[\overline{d}(P_{i,j}\cap Q|Q)=\overline{d}(S_{i,j}\cap Q|Q)=1,\ \forall i\neq j.\]
Then it is easy to see that for every $t>0$
\[(x_i, x_j)\in\mathcal F([\Delta]_t,Q,\overline{\mathcal M}_Q(1)), \quad
(x_i, x_j)\in\mathcal F(X\times X\setminus\overline{[\Delta]_{\delta_{ij}}},Q,\overline{\mathcal M}_Q(1)),\
\forall i\neq j.\]
By Proposition \ref{2}, $C$ is a distributional chaotic set in the sequence $Q$.

For a given $\delta>0$, if $C$ is a Li-Yorke $\delta$-chaotic set,
then we can choose all the above $\delta_{ij}$ being $\delta$,
therefore, $C$ is a distributional $\delta$-chaotic set in the sequence $Q$.
\end{proof}

\begin{rem}
In \cite{HuangYe}, the authors constructed a countable compact space $X$
and a homeomorphism on $X$ with the whole space being a Li-Yorke chaotic set,
then there exists a strictly increasing sequence $Q$ of positive integers
such that $X$ is a distributional chaotic set in the sequence $Q$.
As the method of Theorem 4.3 in \cite{HuangYe}, we can construct a homeomorphism of the Cantor set with the
whole space being a distributional chaotic set in some sequence.
\end{rem}

\begin{thm}\label{thm:li-yorke-mycielski}
Let $(X,f)$ be a TDS and $\delta>0$. Then $(X,f)$ is Li-Yorke $\delta$-chaos if and only if
it is distributional chaos in a sequence.
\end{thm}
\begin{proof}
Let $D$ be an uncountable Li-Yorke $\delta$-chaotic set. Since $X$ is a complete separable metric space,
without lose of generality, assume that $D$ has no isolated points. Put $X_0=\overline D$, then
$D$ is a complete separable metric subspace without isolated points.
Choose a countable dense subset $C$ of $D$. By Theorem \ref{countable},
there exists a strictly increasing sequence $Q$ of positive integers
such that $C$ is a distributional $\delta$-chaotic set in the sequence $Q$.

Denote $E$ be the collection of all distributional $\delta$-scrambled pairs in the sequence $Q$.
By Proposition \ref{Gdeltaset}, $E$ is a $G_\delta$ subset of $X\times X$.
Since $C\times C\setminus \Delta\subset E$ and $C$ is dense in $X_0$, By Mycielski Theorem,
There exists an uncountable Mycielski set $D_0\subset X_0$ such that
$D_0\times D_0\setminus \Delta\subset E$. Thus, $(X,f)$ is distributional chaos in a sequence.
\end{proof}

\begin{rem} (1) For the system $(I,f)$ on the unit closed $I=[0,1]$,
\cite{Smital1989} proved that if $(I,f)$ has a Li-Yorke pair,
then there exists some $\delta>0$ such that it is Li-Yorke $\delta$-chaotic.
Therefore, it is also distributional $\delta$-chaotic in a sequence.

(2) \cite{Blanchard} proved that positive entropy implies Li-Yorke chaos.
By the proof of Theorem 2.3 in \cite{Blanchard}, if $(X,f)$ has positive entropy, then
there exists some $\delta>0$ such that it is Li-Yorke $\delta$-chaotic.
Therefore, it is also distributional $\delta$-chaotic in a sequence.

(3) Let $(X,f)$ be a TDS. The system $(X,f)$ is called {\em topologically transitive},
if for every two nonempty open subsets $U,V\subset X$,
there exists $n\in\mathbb N$ such that $f^n(U)\cap V\neq\emptyset$.
The system $(X,f)$ is called {\em Devaney chaotic}, if it is transitive and has dense periodic points.
\cite{Mai} proved that if $(X,f)$ is Devaney chaotic, then
there exists some $\delta>0$ such that it is Li-Yorke $\delta$-chaotic.
Therefore, it is also distributional $\delta$-chaotic in a sequence.

(4) Let $(X,f)$ be a TDS. The system $(X,f)$ is called {\em topologically weakly mixing},
if the product system $(X\times X, f\times f)$ is transitive.
By the main result of \cite{Xiong1991}, if $(X,f)$ is weakly mixing, then
there exists some $\delta>0$ such that it is Li-Yorke $\delta$-chaotic.
Therefore, it is also distributional $\delta$-chaotic in a sequence.
\end{rem}

\section{Transitive systems}
\begin{de}\cite{Shao}
Let $(X,f)$ be a TDS.
\begin{enumerate}
\item A subset $A$ of $X$ is called {\em uniformly proximal},
if for every $\epsilon>0$, there exists $k\in\mathbb N$
such that $ d(f^{k}(x),f^{k}(y))<\epsilon$ for every $x,y\in A$.
\item A subset $A$ of $X$  is called {\em uniformly rigid},
if for every $\epsilon>0$ there exists $k\in\mathbb N$
such that $d(f^{k}(x),x)<\epsilon$ for every $x\in A$.
\item A subset $A$ of $X$  is called {\em uniformly chaotic},
if there exists a sequence of Cantor sets
$C_1\subset C_2\subset \cdots$ such that $A=\bigcup_{i=1}^{\infty}C_i$ and $C_i$
is both uniformly proximal and uniformly rigid for every $i\geq 1$.
\end{enumerate}
\end{de}

\begin{thm}\label{uniform}
Let $(X,f)$ be a TDS and $A\subset X$.
If $A$ is a uniformly chaotic set,
then there exists a strictly increasing sequence $Q$ of positive integers
such that $A$ is a distributional chaotic set in the sequence $Q$.
\end{thm}
\begin{proof}
By the definition of uniformly chaos, there exists a sequence of Cantor sets
$C_1\subset C_2\subset \cdots$ such that $A=\bigcup_{i=1}^{\infty}C_i$ and $C_i$
is both uniformly proximal and uniformly rigid for every $i\geq 1$.
Then there are two strictly increasing sequences
$P_{N}=\{n_k\}_{k=1}^\infty$ and $S_N=\{m_k\}_{k=1}^\infty$ such that
\[\lim_{k\to\infty}d(f^{n_k^{i,j}}(x), f^{n_k^{i,j}}(y))=0,\quad
\lim_{k\to\infty}d(f^{m_k^{i,j}}(x), f^{m_k^{i,j}}(y))=d(x,y)>0,\ \forall x\neq y\in A_N.\]
By Lemma \ref{SEQ}, there exists a strictly increasing sequence $Q$ such that
\[\overline{d}(P_{N}\cap Q|Q)=\overline{d}(S_{N}\cap Q|Q)=1,\ \forall N\in\mathbb N.\]
Similarly to the proof of Theorem \ref{countable}, we have $A$ is
a distributional chaotic set in the sequence $Q$.
\end{proof}

\begin{thm}\label{Shao}\cite{Shao}
Let $(X,f)$ be a transitive system, where $X$ is a compact metric space without isolated points.
If there exists a subsystem $(Y,f)$ such that $(X\times Y,f\times f)$ is transitive, then
$(X,f)$ has a dense Mycielski uniformly chaotic set.
\end{thm}

\begin{cor}\cite{Shao}\label{shao2}
Let $(X,f)$ be a transitive system, where $X$ is a compact metric space without isolated points.
If the system $(X,f)$ satisfies one of the following conditions:\\
\indent (1) $(X,f)$ is transitive and has a fixed point; \\
\indent (2) $(X,f)$ is totally transitive with a periodic point; \\
\indent (3) $(X,f)$ is scattering; \\
\indent (4) $(X,f)$ is weakly scattering with an equicontinuous minimal subsystem; \\
\indent (4) $(X,f)$ is weakly mixing. \\
Then $(X,f)$ has a dense Mycielski uniformly chaotic set.
Moreover, if $(X,f)$ is transitive and has a periodic point of order $d$,
then there is a closed $f^d$-invariant subset $X_0\subset X$
such that $(X_0, f^d)$ has a dense uniformly chaotic set and $X=\bigcup_{j=0}^{d-1} f^j X_0$;
in particular, $(X,f)$ has a uniformly chaotic set.
\end{cor}

\begin{cor}
Let $(X,f)$ be a TDS. If the system $(X,f)$ satisfies
the condition of Theorem \ref{Shao} or Corollary \ref{shao2}, then $(X,f)$ is distributional chaotic
in a sequence.
\end{cor}

Question: is Li-Yorke chaos equivalent to distributional chaos?

\medskip
\indent {\bf Acknowledge: } The authors would like to thank Prof. Jie L\"u and referees for the
careful reading and helpful suggestions.

\end{document}